\documentclass{amsart}
\usepackage{amssymb}
\usepackage{amsthm}
\usepackage{graphicx}

\newcommand{\pmn}{\par\medskip\noindent}
\newcommand{\pbn}{\par\bigskip\noindent}
\newtheorem{theorem}{Theorem}
\theoremstyle{definition}

\DeclareMathAlphabet{\mathchorus}{OT1}{cmtt}{m}{sl}  

\newcommand{\fps}[1]{\mathchorus{#1}} 
\newcommand{\ix}{\textrm{\textsl{x}}}
\newcommand{\sminus}{\scalebox{0.45}[1.0]{\( - \)}}

\begin{document}
\title{ 
On the Poincar\'{e} functional equation }
\author{Sergei Kazenas }
\begin{abstract}
In this paper, a formula for the solution of the Poincar\'{e} functional equation in algebra of formal power series and its application to continuous iteration  are presented.
\end{abstract}

\email{kazenas@protonmail.com, sergei.kazenas@gmail.com} \maketitle

The Poincar\'{e} functional equation and associated Schr\"{o}der's equation   arise naturally in analytic iteration theory and have been studied by a great many authors (see, for instance, \cite{Der}). In the following, Poincar\'{e}'s equation is treated algebraically; and the formal solution representation is derived using finite symbolic manipulations with homogeneous quantum difference operators.

Here the formal power series treatment is near to that used by Niven \cite{Niv}, but the notation is slightly different.

Let $\mathcal{F}$ denote the algebra of formal power series over field $\mathbb{C}$ of complex numbers. The algebra $\mathcal{P}$ of formal polynomials is a subalgebra of $\mathcal{F}$. Any $\fps{f} \in \mathcal{F} $ is determined  by the vector of its coefficients. Denote by $(\fps{f}_{.j})_{j\geqslant 0}$ or by $ (\fps{f}_{.0}, \fps{f}_{.1}, ...) $ that vector . Also denote by $\ix$ the monomial determined by $\ix_{.j}=\delta_{1,j}$. 

Recall that for any $ \{\fps{f},\fps{g}\} \subset \mathcal{F}$, two basic operations, addition and multiplication, are defined by $(\fps{f}+\fps{g})_{.j}=\fps{f}_{.j}+\fps{g}_{.j}$ and $(\fps{f} \fps{g})_{.j}=\sum_{0\leqslant l \leqslant j}\fps{f}_{.l}\fps{g}_{.j-l}$, accordingly. If $\fps{f}$ is not polynomial, then the composition is generally defined by $(\fps{f} \circ \fps{g})_{.j} =\sum_{0\leqslant l \leqslant j} \fps{f}_{.l}(\fps{g}^l)_{.j}$ only when $\fps{g}_{.0}=0$. If $\fps{f}$ is polynomial, there are no restrictions on $\fps{g}$.

In addition, it is convenient to implement one outer operation; for any element $a$ of some  algebra $\mathcal{A}$, which is  isomorphic  to $\mathcal{F}$ or $\mathcal{P}$, define $\fps{f}[a]:=  \sum_{j \geqslant 0} \fps{f}_{.j} a^j$ .  For example, $\fps{f}[\ix \fps{g}]\equiv \fps{f}\circ (\ix \fps{g})$, when $\{\fps{f},\fps{g}\} \subset  \mathcal{F}$. If $a$ is number, then $\fps{f}[a]$ may only be a number. The  composition operator $\widehat{q\ix}$ defined by $\widehat{q\ix}\fps{f} = \fps{f} \circ (q\ix)$, where $q \in \mathbb{C} $,   will also be needed later on. Clearly, the expression $\fps{p}[\widehat{q\ix}]$ represents a linear operator, when $\fps{p}$ is polynomial.

Recall that $\fps{f}$ is invertible if and only if $\fps{f}_{.0}=0$ and $\fps{f}_{.1}\neq 0$, or $\fps{f}_{.0} \neq 0$ and $\fps{f}=\fps{f}_{.0}+\fps{f}_{.1} \ix$. Let $\fps{f} \in \ix \mathcal{F}$ be invertible and $q \neq 0$ not be a root of unity. Let $\fps{f}$ also satisfy Poincar\'{e}'s equation
$$ 
\fps{f} \circ (q\fps{x}) = \fps{p} \circ \fps{f} \, \text{.}
\eqno{(1)}
$$
Then it immediately follows that $\fps{f}^{\circ\sminus 1}$ satisfies Schr\"{o}der's equation
$q \fps{f}^{\circ\sminus 1} = \fps{f}^{\circ\sminus 1} \circ \fps{p} $.

It also follows that an invertible solution exists if and only if $\fps{p}_{.0} = 0$ and $\fps{p}_{.1} = q$. Therefore, equation (1) can be rewritten in the equivalent form: $\fps{f} \circ (q^{\sminus 1} \ix) = \fps{p}^{\circ \sminus 1} \circ \fps{f}$. Without loss of generality, the nonzero coefficient $\fps{f}_{.1}$ can be fixed by $\fps{f}_{.1}=1$ and the other coefficients $\fps{f}_{.2}, \fps{f}_{.3}, ... $ can be calculated recursively from
$$
\fps{f}_{.j} = \frac{1}{q^j-1}\sum_{l=1}^{j-1}\fps{p}_{.l} (\fps{f}^l)_{.j} \, \text{.} \eqno{(2)}
$$

However, there are  approaches to represent the coefficients of $\fps{f}$ nonrecursively. One such approach is described below.

For an arbitrary nonzero polynomial $\fps{g}$ define a homogeneous quantum difference operator in $\mathcal{F}$ by

$$D_{\fps{g};q} := (\ix^{\sminus n}) \fps{g}[\widehat{q\ix}]  \, \text{,}$$
where operator $(\ix^{\sminus n})$ is naturally defined in $\ix^n {\mathcal{F}}$, and $n$ is the smallest nonnegative integer such that $\fps{g}[q^n] \neq 0$; let call that number the order of operator $D_{\fps{g};q}$ or $q$-difference order of $\fps{g}$. 
For example, $D_{\frac{1-\ix}{1-q};q}$ is the first order operator, which can be viewed as Euler-Jackson difference operator in ${\mathcal{F}}$. 

From $D_{\fps{g};q}\ix^m = (\ix^{\sminus n}) \fps{g}[\widehat{q\ix}] \ix^m = (\ix^{\sminus n}) \fps{g}[q^m]\ix^m = 0$ (0 $\leqslant m<n$) Macloren-like expansion comes:

$$ \fps{f}_{.j} = \frac{(D_{\fps{g}_j;q}\fps{f})[0]}{\fps{g}_j[q^j]}  \text{ ,} \eqno{(3)}
$$

where $D_{\fps{g}_j;q}$ is of order $j$.

Now the preparation is complete to derive a nonrecursive companion to (2).

\begin{theorem}
Let $\fps{f}$ satisfy Poincar\'{e}'s equation {\normalfont{(1).}}  $\fps{f}_{.0}=0$, $\fps{f}_{.1}=1$; $\fps{g}_j$ stands for an arbitrary nonzero polynomial of $j$-th $q$-difference order, then the coefficients of $\fps{f}$ can be calculated by
$$
\fps{f}_{.j} = \frac{1}{\fps{g}_j[q^j]}\sum_{i\geqslant 0} \fps{g}_{j.i}(\fps{p}^{\circ i})_{.j} \, \text{.} \eqno{(4)}
$$
\end{theorem}
\begin{proof}
Calculate $q$-differences: 
$$
D_{\fps{g}_j;q}\fps{f} = (\ix^{\sminus j})\fps{g}_j[\widehat{q\ix}]\fps{f}=
\ix^{\sminus j} \sum_{i\geqslant 0}\fps{g}_{j.i}\fps{f}[q^i\ix]=
\ix^{\sminus j} \sum_{i\geqslant 0}\fps{g}_{j.i}\fps{p}^{\circ j}[\fps{f}] \, \text{.}
$$
Note that $\frac{\fps{f} - \ix}{\ix}[0]=0$ and continue:
$$
\begin{aligned}
(D_{\fps{g}_j;q}\fps{f})[0] & =
\left( \ix^{\sminus j} \sum_{i\geqslant 0}\fps{g}_{j.i}\fps{p}^{\circ j} \left[  \ix \left( 1+\frac{\fps{f}-\ix}{\ix} \right) \right] \right) [0] \\
& =\left( \ix^{\sminus j} \sum_{i\geqslant 0}\fps{g}_{j.i}\fps{p}^{\circ j}[\ix] \right) [0] 
=\sum_{i\geqslant 0}\fps{g}_{j.i}(\fps{p}^{\circ i})_{.j} \, \text{.}
\end{aligned}
$$

Taking in account (3) completes the proof.
\end{proof}

\pmn \emph{Remark.} Suppose $\fps{p}$ is polynomial and the function $f$ defined by $f(x) = \sum_{j \geqslant 0} \fps{f}_{.j} x^j$ is entire. Then $f$ is determined by its values at points $a q^0, a q^1, ...$, where $a$ is such that $f(a)$ is not periodic (or eventually periodic) point of the map $x \mapsto \fps{p}[x]$ . This is clear from the fact that sums in (4) are finite and $\fps{p}^{\circ i}$ are determined by its values at points $\fps{f}[a q^0], \fps{f}[a q^1], ...$
\pmn

As a corollary, one can derive a formula for continuous iteration.

Fix $j$. Using Kac and Cheung's notation, $(x - a)^j_q:=(x-a)(x-q a)...(x-q^{j-1}a)$, and taking $\ix^{n-j} (\ix - 1)^j_q$ or $\ix^{n-1-j} (\ix - 1)^j_q$ ($n > j$) in place of $\fps{g}_j$ in (4) gives equality:

$$
\begin{aligned}
& q^{ \sminus (n-j)j}\sum_{i=n-j}^n (\ix^{n-j} (\ix - 1)^j_q)_{.i} (\fps{p}^{\circ i})_{.j} 
\\
& = q^{ \sminus (n-1-j)j}\sum_{i=n-1-j}^{n-1} (\ix^{n-1-j} (\ix - 1)^j_q)_{.i} (\fps{p}^{\circ i})_{.j}
\end{aligned}
$$
Using Gauss' binomial, this yelds recursive formula:
$$
\begin{aligned}
(\fps{p}^{\circ n})_{.j} =& \sum_{i=1}^{j+1}  \binom{j+1}{i}_q
(\sminus 1)^{i+1}q^{i(i-1)/2} (\fps{p}^{\circ (n-i)})_{.j}  \\
& 
= (S^0 - (S^0-S)^j_q)(\fps{p}^{\circ n})_{.j}
\end{aligned}
$$
where $S$ is shift operator with respect to $n$, i.e. $S(\fps{p}^{\circ n})_{.j} = (\fps{p}^{\circ (n-1)})_{.j} $. 

It follows from this representation that $(\fps{p}^{\circ n})_{.j}$ is a polynomial of degree $j$ in $q^n$ (recall that $q = \fps{p}_{.1}$) and Waring's interpolation process gives:
$$
(\fps{p}^{\circ n})_{.j}=\sum_{l=0}^j\frac{(q^n-1)^l_q (q^n-q^{l+1})^{j-l}_q}{(q^l-1)^l_q (q^l-q^{l+1})^{j-l}_q}(\fps{p}^{\circ l})_{.j} \eqno{(5)}
$$
Expanding this polynomial gives:
$$
(\fps{p}^{\circ n})_{.j}=\sum_{i=1}^jq^{n i} \sum_{k=0}^i ( \sminus 1)^{j-k} q^{(j-k)(j-k+1)/2}\binom{j+1}{k}_q \rho_{l,j,i} 
\text{ ,}
$$
where
$$
\rho_{l,j,i}:=
\sum_{l=0}^j\frac{q^{k l - i l -l}(\fps{p}^{\circ l})_{.j}}{(q^l-1)^l_q (q^l-q^{l+1})^{j-l}_q}
\text{ .}
$$
The behavior of $\rho_{l,j,i}$ may be the subject of further research.

\pmn

Jacobi [6] gave an explicit general formula for $(\fps{p}^{\circ n})_{.j}$ (for positive integer $n$). Therefore, formula (5) allows calculating fractional functional powers directly. And there are no restrictions on $q$. For example,
$$
(q \ix + q {\ix}^2 )^{\circ \frac{1}{2}} = 
q^{1/2} \ix +\frac{q^{1/2}}{1+q^{1/2}} \ix ^2
- \frac{2q}{(1+q^{1/2})^2 (1+q)} \ix^3 + ...
$$

Of course, analytic function $x \mapsto q x +q x^2$ ($q \neq 1$) has two functional square roots. The second root can be obtained by replacing "$q^n$" with "$ \sminus q^{1/2}$" in the expression (5). (In case $q = 1$, it has only one functional square root. This case was developed in \cite{Lab}, for instance.)

Note that changing $\fps{g}_j$ in the expansion (4) does not lead to new properties of  iterated formal power series $\fps{p}^{\circ n}$. Therefore, by setting $\fps{g}_j = (\ix - 1)^j_q$ and using Gauss' binomial, it may be simplified:

$$
\fps{f}_{.j}=\frac{1}{(q-1)^j q^{(j-1)j/2}j!_q}\sum_{i=0}^j q^{(j-i)(j-i-1)/2}\binom{j}{i}_q ( \sminus 1)^{j-i} (\fps{p}^{\circ i})_{.j}
\text{ ,}
$$

where
$$
j!_q := \frac{(1-q)_q^j}{(1-q)^j}
\text{ .}
$$

\pbn

\vspace{5mm}
\end{document}